\documentclass{gen-j-l}
\usepackage{amssymb}
\usepackage{amsfonts}

\newtheorem{theorem}{Theorem}[section]
\newtheorem{lemma}[theorem]{Lemma}
\theoremstyle{definition}

\newtheorem{example}[theorem]{Example}

\theoremstyle{remark}
\newtheorem{remark}[theorem]{Remark}
\numberwithin{equation}{section}
\theoremstyle{plain}

\newtheorem{corollary}{Corollary}

\newtheorem{proposition}{Proposition}

\copyrightinfo{2001}{enter name of copyright holder}
\input{tcilatex}

\begin{document}
\title[$v$-domains]{On super $v$-domains}
\author{M. Zafrullah}
\address{Department of Mathematics, Idaho State University,\\
Pocatello, Idaho, USA}
\email{mzafrullah@usa.net}
\urladdr{}
\thanks{}

\begin{abstract}
An integral domain $D,$ with quotient field $K,$ is a $v$-domain if for each
nonzero finitely generated ideal $A$ of $D$ we have $(AA^{-1})^{-1}=D.$ It
is well known that if $D$ is a $v$-domain$,$ then some quotient ring $D_{S}$
of $D$ may not be a $v$-domain. Calling $D$ a super $v$-domain if every
quotient ring of $D$ is a $v$-domain we characterize super $v$-domains as
locally $v$-domains. Using techniques from factorization theory we show that 
$D$ is a super $v$-domain if and only if $D[X]$ is a super $v$-domain if and
only if $D+XK[X]$ is a super $v$-domain and give new examples of super $v$%
-domains that are strictly between $v$-domains and P-domains that were
studied in [Manuscripta Math. 35(1981)1-26]
\end{abstract}

\maketitle

An integral domain $D,$ with quotient field $K,$ is called a $v$-domain if
for every finitely generated nonzero ideal $A$ of $D,$ $A$ is $v$%
-invertible, i.e., we have $(AA^{-1})^{-1}=D$ or equivalently $%
(AA^{-1})_{v}=D.$ Now $v$-domains, the oldest known notion in,
multiplicative ideal theory, according to \cite{FZ v}, come defined in
various ways. They are called $\ast $-Prufer if for some star operation,
every finitely generated nonzero ideal $A$ is $\ast $-invertible, i.e., we
have $(AA^{-1})^{\ast }=D$ (see, e.g., \cite{AAFZ v}). As is apparent from
the above definitions, $v$-domains are modeled after Prufer domains.
Mimicking the proof, for Prufer domain, by Prufer himself, it was shown in 
\cite{MNZ} that $D$ is a $v$-domain if and only if every two generated
nonzero ideal of $D$ is $v$-invertible (see also an earlier paper by Gabelli 
\cite{Gab} that hints at the possibility.) Call $D$ essential if $D$ has a
family $\mathcal{F}$ of prime ideals such that $D_{P}$ is a valuation domain
for each $P\in \mathcal{F}$ and $D=\cap _{P\in \mathcal{F}}D_{P}.$ As
indicated in \cite{FZ v} an essential domain is a $v$-domain and so is the
so-called "P-domain". A P-domain here is an essential domain, each of whose
quotient rings is essential. The P-domains were initially studied in \cite%
{MZ}. It turns out, however, that if $D$ is a $v$-domain and $A$ a nonzero
finitely generated ideal of $D$, $A^{-1}$ may not even be close to being
finitely generated, completely unlike Prufer domains. Also, every quotient
ring of a Prufer domain is Prufer. Yet using an example of Heinzer's, \cite%
{He}, of an essential domain with a non-essential quotient ring that cannot
be a $v$-domain, one can show that if $D$ is a $v$-domain and $S$ a
multiplicative set of $D,$ then $D_{S}$ need not be a $v$-domain, see
section 3 of \cite{FZ v} for a discussion on this example. This raises the
questions: (a) if $D$ is a $v$-domain, under what conditions on a
multiplicative set $S,$ or on $D,$ can we be sure that $D_{S}$ is a $v$%
-domain? (b) what are the $v$-domains whose quotient rings are also $v$%
-domains and that are not any of the known examples of $v$-domains all of
whose quotient rings are $v$-domains and (c) if every proper quotient ring
of $D$ is a $v$-domain, must $D$ be a $v$-domain? The purpose of this note
is to start a discussion on these questions. For our part we characterize
super $v$-domains i.e. domains whose quotient rings are all $v$-domains and
discuss some conditions that will ensure that a quotient ring of a $v$%
-domain is a $v$-domain. We show for instance that if $D$ is a $v$-domain
and $S$ is a splitting or a $t$-splitting set of $D$ then $D_{S}$ is a $v$%
-domain. Using some of these results we give an example schema for super $v$%
-domains, showing that these super $v$-domains are strictly between $v$%
-domains and P-domains, we show that $D$ is a super $v$-domain if and only
if $D+XK[X]$ is a super $v$-domain. We also show that if $X$ is an
indeterminate over $D$, then $D$ is a super $v$-domain if and only if $D[X]$
is. (The answer to question (c) is that for a one dimensional quasi local
domain $D$ a proper quotient ring is the field of fractions of $D$ and hence
a $v$-domain. But a one dimensional quasi local domain need not be a $v$%
-domain.)

It seems pertinent to let the reader in on the terminology that we have used
above and that we are going to use when we prove our results. Let $D$ be an
integral domain with quotient field $K$ and let $F(D)$ be the set of nonzero
fractional ideals of $D.$ A star operation is a function $A\mapsto A^{\ast }$
on $F(D)$ with the following properties:

If $A,B\in F(D)$ and $a\in K\backslash \{0\},$ then

(i) $(a)^{\ast }=(a)$ and $(aA)^{\ast }=aA^{\ast }.$

(ii) $A\subseteq A^{\ast }$ and if $A\subseteq B,$ then $A^{\ast }\subseteq
B^{\ast }.$

(iii) $(A^{\ast })^{\ast }=A^{\ast }.$

We may call $A^{\ast }$ the $\ast $-image ( or $\ast $-envelope ) of $A.$ An
ideal $A$ is said to be a $\ast $\textit{-ideal} if $A^{\ast }=A.$ Thus $%
A^{\ast }$ is a $\ast $-ideal (by (iii)). Moreover (by (i)) every principal
fractional ideal, including $D=(1)$, is a $\ast $- ideal for any star
operation $\ast $.

For all $A,B\in F(D)$ and for each star operation $\ast $, $(AB)^{\ast
}=(A^{\ast }B)^{\ast }=(A^{\ast }B^{\ast })^{\ast }$. These equations define
what is called $\ast $-multiplication \textit{( or }$\ast $-product)\textit{%
. }

Define $A_{v}=(A^{-1})^{-1}$ and $A_{t}=\bigcup \{J_{v}|$ $0\neq J$ is a
finitely generated subideal of $A\}.$ The functions $A\mapsto A_{v}$ and $%
A\mapsto A_{t}$ on $F(D)$ are more familiar examples of star operations
defined on an integral domain. A $v$-ideal is better known as a divisorial
ideal. The identity function $d$ on $F(D)$, defined by $A\mapsto A$ is
another example of a star operation. There are of course many more star
operations that can be defined on an integral domain $D$. But for any star
operation $\ast $ and for any $A\in F(D),$ $A^{\ast }\subseteq A_{v}.$ Some
other useful relations are: For any $A\in F(D),$ $(A^{-1})^{\ast }=$ $%
A^{-1}=(A^{\ast })^{-1}$ and so, $(A_{v})^{\ast }=A_{v}=(A^{\ast })_{v}.$
Using the definition of the $t$-operation one can show that an ideal that is
maximal w.r.t. being a proper integral $t$-ideal is a prime ideal of $D$,
each ideal $A$ of $D$ with $A_{t}\neq D$ is contained in a maximal $t$-ideal
of $D$ and $D=\cap D_{M},$ where $M$ ranges over maximal $t$-ideals of $D.$
For more on $v$- and $t$-operations the reader may consult sections 32 and
34 of Gilmer \cite{Gil}. Our terminology essentially comes from \cite{Gil}.

Call a multiplicative set $S$ of $D$ a splitting set if $S$ is saturated and
for each $d\in D\backslash \{0\}$ we can write $d=d^{\prime }s$ where $s\in
S $ and $d^{\prime }\in D$ such that $(d^{\prime },t)_{v}=D$ for all $t\in
S. $ For more on splitting sets look up \cite{AAZ spl}. On the other hand a
multiplicative set $S$ of $D$ is a $t$-splitting set if for all $d\in
D\backslash \{0\}$ we can write $dD=(AB)_{t}$ where $B_{t}\cap S\neq \phi $
and $(A,s)_{v}=D$ for all $s\in S.$ The $t$-splitting sets were introduced
and applied in \cite{AAZ t-spl}.

Let's call $D$ a super $v$-domain if every quotient ring of $D$ is a $v$%
-domain. Let us be clear about what we are looking for, when we study "super 
$v$-domains" as there do exist super $v$-domains in the form of the
P-domains and Prufer domains and the so-called Prufer $v$-Multiplication
domains or PVMDs. PVMDs, by the way, are $v$-domains such that $aD\cap
bD=A_{v}$ for some finitely generated ideal $A,$ for all $a,b\in D\backslash
\{0\}$, \cite{MZ}. According to \cite{MZ} a PVMD is a P-domain. In our study
of super $v$-domains we are looking for $v$-domains $D$ that are not
P-domains yet have the property that $D_{S}$ is a $v$-domain for each
multiplicative set $S$ of $D.$ In other words we are looking for $v$-domains 
$D$ that lie strictly between $v$-domains and P-domains, with the property
that every quotient ring of $D$ is a $v$-domain.

The first thing that seems to prevent a $v$-domain from having a quotient
ring that is a $v$-domain seems to be that while for a nonzero finitely
generated ideal $I$ we have $(ID_{S})^{-1}=I^{-1}D_{S}$ we have no such
general formula for a nonzero ideal $I$. One way of dealing with a situation
like this is to bring in a new definition. Call a quotient ring $D_{S}$ of $%
D $ super extending if for each nonzero ideal $I$ of $D$ we have $%
(ID_{S})^{-1}=I^{-1}D_{S}.$ An immediate consequence is that if $D_{S}$ is
super extending, then $(ID_{S})_{v}=I_{v}D.$

\begin{lemma}
\label{Lemma A} If $D_{S}$ is super extending and $D$ is a $v$-domain, then $%
D_{S}$ is a $v$-domain.
\end{lemma}

\begin{proof}
Let $\alpha ,\beta \in D_{S}$. Then $\alpha =\frac{a}{s},\beta =\frac{b}{t}$
for some $a,b\in D$ and $s,t\in S$ and $(\alpha ,\beta )D_{S}((\alpha ,\beta
)D_{S})^{-1}=(a,b)D_{S}((a,b)D_{S})^{-1}=((a,b)(a,b)^{-1})D_{S}.$ Now as $%
D_{S}$ is super extending we conclude that $((\alpha ,\beta )D_{S}((\alpha
,\beta
)D_{S})^{-1})^{-1}=(((a,b)(a,b)^{-1})D_{S})^{-1}=(((a,b)(a,b)^{-1}))^{-1}D_{S}=D_{S} 
$ because in $D$ we have $(((a,b)(a,b)^{-1}))^{-1}=D.$
\end{proof}

But the drawback of Lemma \ref{Lemma A} is that if $D_{S}$ happens to be
such that $(a,b)^{-1}D_{S}$ is a finitely generated ideal of $D_{S}$ for
each pair $a,b$ of $D$, then Lemma \ref{Lemma A} would be an overkill.
Though $D_{S}$ would have to be a stronger form of a PVMD. All this beside,
super extending is too much even for our needs. So let's call $D_{S}$ simple
extending if $(((a,b)(a,b)^{-1})D_{S})^{-1}=(((a,b)(a,b)^{-1}))^{-1}D_{S}$.
We do seem to have disadvantages of super extending when working with simple
extending and simple extending is sort of too obvious a ploy, but it may
work in some interesting ways neatly.

\begin{proposition}
\label{Proposition B}Let $D$ be an integral domain and let $\{S_{\alpha }\}$
be a family of multiplicative sets of $D$ such that $D=\cap D_{S_{\alpha }}.$
If, for each $\alpha \in I,$ $D_{S_{\alpha }}$ is a simple extending
quotient ring of $D$ and a $v$-domain, then $D$ is a $v$-domain.
\end{proposition}

\begin{proof}
Note that, as the inverse of an ideal is divisorial, we have $%
(((a,b)(a,b)^{-1}))^{-1}=\cap (((a,b)(a,b)^{-1}))^{-1}D_{S_{\alpha }}$ $%
=\cap ((((a,b)(a,b)^{-1}))D_{S_{\alpha }})^{-1}=\cap D_{S_{\alpha }}=D.$
\end{proof}

But there is a better result available on the market in the form of
Proposition 3.1 of \cite{FZ v}. This result says.

\begin{proposition}
\label{Proposition C}Let $\{D_{\lambda }|\lambda \in \Lambda \}$ be a family
of flat overrings of D such that $D=$ $\cap _{\lambda \in \Lambda }$ $%
D_{\lambda }$ . If each of $D_{\lambda }$ is a $v$-domain, then so is $D$.
\end{proposition}

Let us recall that a prime ideal $P$ is called an associated prime of a
principal ideal $(a)$ if $P$ is minimal over an ideal of the form $0\neq
(a):(b)=\{r\in D|rb\in (a)\}\neq D.$ Associated primes of principal ideals,
or simply associated primes, of $D$ have been studied by quite a few
authors, but our reference in this regard is \cite{BH}. According to
Proposition 4 of \cite{BH}, if $S$ is a multiplicative set of $D$ and $%
\{P_{\alpha }\}$ is the family of associated primes of principal ideals of $D
$ disjoint from $S,$ then $D_{S}=\cap _{\alpha }D_{P_{\alpha }}.$

With Proposition \ref{Proposition C} at hand, we can state and prove the
following characterization of super $v$-domains.

\begin{theorem}
\label{Theorem D}(\cite[Proposition 3.4]{FZ v})The following are equivalent
for an integral domain $D.$ (1) $D_{S}$ is a $v$-domain for every
multiplicative set $S$ of $D,$ (2) $D_{M}$ is a $v$-domain for every maximal
ideal $M$ of $D,$ (3) $D_{P}$ is a $v$-domain for every prime ideal $P$ of $%
D $ and (4) $D_{P}$ is a $v$-domain for every associated prime $P$ of $D.$
\end{theorem}

\begin{proof}
That (1) $\Rightarrow $ (2) $\Rightarrow $ (3) $\Rightarrow $ (4) is
obvious. For (4) $\Rightarrow $ (1), let $S$ be a multiplicative set of $D$
and let $\mathcal{F}=\{P_{\alpha }\}$ be the family of associated primes
disjoint from $S.$ Then by (4) each of $D_{P_{\alpha }}$ is a $v$-domain and
by \cite[Proposition 4]{BH} $D_{S}=\cap _{P_{\alpha }\in \mathcal{F}%
}D_{P_{\alpha }}.$ Thus by Proposition \ref{Proposition C}, $D_{S}$ is a $v$%
-domain.
\end{proof}

There is, however, a situation in which $D_{S}$ is a $v$-domain, whenever $D$
is. That is when the multiplicative set $S$ in $D$ is a splitting set. If $S$
is a splitting set, the set $T=\{t\in D|$ $(t,s)_{v}=D$ for all $s\in S\}$
often denoted as $S^{\bot }$ is called the $m$-complement of $S.$ Indeed if $%
S$ is a splitting set and $T=S^{\bot },$ then $D=D_{S}\cap D_{T}$ and $%
dD_{S}\cap D=tD$ where $t\in T$ such that $d=ts$ for some $s\in S.$ A
splitting set $S$ of $D$ is an lcm splitting set if $sD\cap xD$ is principal
for all $s\in S$ and for all $x\in D\backslash \{0\}.$

\begin{theorem}
\label{Theorem E}Let $S$ be a splitting multiplicative set of $D$ and let $%
T=S^{\bot }.$ If $D$ is a $v$-domain, then so is $D_{S}.$ Moreover if $S$ is
an lcm splitting set then $D_{S}$ is a $v$-domain if and only if $D$ is a $v$%
-domain.
\end{theorem}

\begin{proof}
Suppose that $D_{S}$ is not a $v$-domain. That is, there is a pair $a,b$ of $%
D_{S}$ such that $(((a,b)(a,b)^{-1})$ $D_{S})_{v}\neq D_{S}.$ Since $%
(r,s)^{-1}D_{S}=((r,s)D_{S})^{-1},$ for $r,s\in D\backslash \{0\},$ we can
take $a,b\in D$ and regard $(a,b)(a,b)^{-1}$ as an ideal of $D.$ Since $%
(((a,b)(a,b)^{-1})$ $D_{S})_{v}\neq D_{S}$, $(a,b)(a,b)^{-1}\cap S=\phi .$
Again since $(((a,b)(a,b)^{-1})$ $D_{S})_{v}\neq D_{S}$ there exist $x,y\in
D_{S}$ such that $((a,b)(a,b)^{-1})$ $D_{S}\subseteq \frac{x}{y}D_{S}$ where 
$x\nmid y$ in $D_{S}.$ As $S$ is a splitting set, we can take $x,y\in T.$
But then $y((a,b)(a,b)^{-1})$ $D_{S}\subseteq xD_{S}$ and $%
y((a,b)(a,b)^{-1}) $ $\subseteq y((a,b)(a,b)^{-1})$ $D_{S}\cap D\subseteq
xD_{S}\cap D.$ As $x\in T,$ we have $xD_{S}\cap D=xD$ (\cite{AAZ spl},
Theorem 2.2). Thus we have $y((a,b)(a,b)^{-1})\subseteq xD.$ Applying the $v$%
-operation throughout and noting that $D$ is a $v$-domain we conclude that $%
yD\subseteq xD.$ But then $yD_{S}\subseteq xD_{S},$ a contradiction. Whence $%
D_{S}$ is a $v$-domain. For the moreover part note that $D=D_{S}\cap D_{T}$
where $D_{T}$ is a GCD domain, by Theorem 2.4 of \cite{AAZ spl}. Thus if $S$
is lcm splitting $D_{S}$ is a $v$-domain and so is $D_{T},$ being a GCD
domain, forcing $D=D_{S}\cap D_{T}$ to be a $v$-domain, by Proposition \ref%
{Proposition C}.
\end{proof}

\begin{theorem}
\label{Theorem F}Let $D$ be an integral domain with quotient field $K$ and
let $X$ be an indeterminate over $D.$ Then $D$ is a super $v$-domain if and
only if $D+XK[X]$ is a super $v$-domain.
\end{theorem}

\begin{proof}
Let $D$ be a super $v$-domain. Then by Theorem 4.42 of \cite{CMZ 1} $%
T=D+XK[X]$ is a $v$-domain. Also by Proposition 2.2 of \cite{CMZ 2}, every
overring $S$, and hence every quotient ring $S$, of $T$ is a quotient ring
of $S\cap K+XK[X].$ According to the proof of Proposition 2.2 of \cite{CMZ 2}
the elements of $S$ are of the form $\frac{\alpha +Xf(X)}{1+Xg(X)}$ where $%
\alpha \in S\cap L.$ Let $U=\{u\in D|u$ is a unit in $S\}.$ Then $%
D_{U}\subseteq S\cap K.$ Let $h\in S.$ Then $h=\frac{a+Xf(X)}{b+Xg(X)}$
where, $a,b\in D$ and, $b+Xg(X)$ is a unit in $S.$ This gives $b=b(1+\frac{X%
}{b}g(X)(1+\frac{X}{b}g(X)^{-1}$ and so $b$ is a unit in $S\cap K$, whence $%
b\in U.$ But then $a/b=h(0)\in D_{U}.$ Noting that $h(0)\in S\cap K$ we
conclude that $D_{U}=S\cap K.$ This leads to the conclusion that $S$ is a
quotient ring of $D_{U}+XK[X].$ Since $D$ is a super $v$-domain $D_{U}$ is a 
$v$-domain and so is $D_{U}+XK[X].$ Next, by the proof of Proposition 2.2 of 
\cite{CMZ 2}, denoting by $U(S)$ the set of units of $S$ we have $%
U(S)=\{f\in D_{U}+XK[X]|f=u+Xg(X),$ where $u$ is a unit in $D_{U}\}$ and as
elements of the form $1+Xg(X)$ are finite products of height one primes of $%
D_{U}+XK[X]$ (\cite{CMZ 1}, Theorem 4.21) we conclude that $U(S)$ is a
splitting set generated by primes. But then, by Theorem \ref{Theorem E}, $%
S=(D_{U}+XK[X])_{U(S)}$ is a $v$-domain. For the converse note that if $T$
is a multiplicative set in $D$, then $(D+XK[X])_{T}=D_{T}+XK[X]$ which is a $%
v$-domain if and only if $D_{T}$ is a $v$-domain. Thus if $D+XK[X]$ is a
super $v$-domain, then so is $D.$
\end{proof}

Some super $v$-domains such as the P-domains have the property that $D_{P}$
is a valuation domain for every associated prime of a principal ideal of $D$%
. Now if $P$ is an associated prime of a principal ideal, one can easily
show that $D_{P}$ is $t$-local, i.e., $PD_{P}$ is a $t$-ideal \cite{FZ v}.
This may lead one to ask if a $t$-local super $v$-domain is close to a
valuation domain. The answer is: Close but not too close, as there does
exist a one dimensional completely integrally closed integral domain $%
\mathcal{N}$, due to Nagata \cite{Nag 1} and \cite{Nag 2}, that is not a
valuation domain and a one dimensional quasi local domain is $t$-local. (Of
course a completely integrally closed domain is a $v$-domain.) Now,
trivially, $\mathcal{N}$ has the property that every quotient ring of $%
\mathcal{N}$ is $\mathcal{N}$ or $qf(\mathcal{N)}$. Thus, albeit trivially, $%
\mathcal{N}$ serves as an example of a super $v$-domain. This gives us the
following example.

\begin{example}
\label{Example G} Let $F$ be the quotient field of $\mathcal{N}$ and let $X$
be an indeterminate on $F.$ Then $\mathcal{N+}XF[X]$ is a super $v$-domain.
\end{example}

Illustration: By Theorem \ref{Theorem F}, every quotient ring $S$ of $%
\mathcal{N+}XF[X]$ is a quotient ring $(\mathcal{N+}XF[X])_{U}$ of $\mathcal{%
N+}XF[X],$ by a multiplicative set $U$ generated by elements of the form $%
1+Xg(X),$ or a quotient ring of $F[X].$ Since $\mathcal{N+}XF[X]$ is a $v$
domain and elements of the form $1+Xg(X)$ being products of height one
primes, $U$ is a splitting set and by Theorem \ref{Theorem E}, $(\mathcal{N+}%
XF[X])_{U}$ is a $v$-domain. Also since $F[X]$ is a PID every quotient ring
of $F[X]$ is a PID and hence a $v$-domain. So, every quotient ring of $%
\mathcal{N+}XF[X]$ is indeed a $v$-domain.

Indeed $\mathcal{N+}XF[X]$ provides a "non-trivial" example of a super $v$%
-domain and Theorem \ref{Theorem F} provides a scheme for producing super $v$%
-domains of any Krull dimension.

Next call a domain $D$ a $v$-local domain if $D$ is quasi local such that
the maximal ideal $M$ of $D$ is divisorial. Of course, the situation can
drastically change if we relax "$t$-local" to "$v$-local".

\begin{proposition}
\label{Proposition H}An integral domain $D$ is a $v$-local $v$-domain if and
only if $D$ is a valuation domain with maximal ideal $M$ principal.
\end{proposition}

\begin{proof}
Let $D$ be a $v$-local $v$-domain and let $A$ be a nonzero finitely
generated ideal of $D.$ Then $AA^{-1}=D.$ For if $AA^{-1}\neq D$ we must
have $AA^{-1}\subseteq M.$ But as $M$ is a $v$-ideal and $D$ a $v$-domain we
have $D=(AA^{-1})_{v}\subseteq M_{v}=M$ a contradiction. Whence every
nonzero finitely generated ideal of $D$ is invertible and hence principal,
because $D$ is $v$-local and hence quasi local. Thus $D$ is a valuation
domain. Now the maximal ideal being divisorial means $M_{v}\neq D$ which
means that there is a pair of elements $a,b$ of $D$ such that $M\subseteq
(a/b)D$ where $a\nmid b.$ Since $a\nmid b$ and $D$ is a valuation domain $%
M\subseteq (a/d)D$ a principal ideal of $D$. But then $M$ is principal
because $M$ is the maximal ideal. The converse is obvious.
\end{proof}

Let's recall from Griffin \cite[Theorem 5]{Gri} that $D$ is a PVMD if and
only if for every finitely generated nonzero ideal $I$ of $D$ we have $%
(II^{-1})_{t}=D$ if and only if $D_{P}$ is a valuation ring for every
maximal $t$-ideal of $D.$

\begin{corollary}
\label{Corollary J}Let $D$ be locally a $v$-domain. Suppose that for every
maximal $t$-ideal $M$ of $D$ we have $MD_{M}$ divisorial then $D$ is a PVMD.
\end{corollary}

\begin{proof}
For every maximal $t$-ideal $M$ we have $D_{M}$ a $v$-domain and $MD_{M}$ a
divisorial ideal. Then by Proposition \ref{Proposition H} we have that $%
D_{M} $ is a valuation domain with maximal ideal principal.

Alternative proof: Let $J$ be a nonzero ideal of $D.$ We claim that $JJ^{-1}$
is not in any maximal $t$-ideal of $D.$ For if $JJ^{-1}\subseteq M.$ Then $%
(JJ^{-1})D_{M}=JD_{M}J^{-1}D_{M}=JD_{M}(JD_{M})^{-1}\subseteq MD_{M}.$ Since 
$D_{M}$ is a $v$-domain, $D_{M}=((JD_{M}(JD_{M})^{-1})_{v}.$ Yet as $MD_{M}$
is divisorial and $JD_{M}J^{-1}D_{M}=JD_{M}(JD_{M})^{-1}\subseteq MD_{M}$ we
get $D_{M}=((JD_{M}(JD_{M})^{-1})_{v}\subseteq MD_{M}$ a contradiction. Now $%
JJ^{-1}$ not being in any maximal $t$-ideals means that $(JJ^{-1})_{t}=D.$
Thus every nonzero finitely generated ideal of $D$ is $t$-ivertible and this
is another characteristic property of PVMDs.
\end{proof}

Recall that a prime ideal $P$ of a domain $D$ is called strongly prime if $%
x,y\in K$ and $xy\in P$ imply that $x\in P$ or $y\in P.$ According to \cite%
{HH}, $D$ is a pseudo valuation domain PVD if every prime ideal of $D$ is
strongly prime. It turns out that a PVD is a valuation domain or a quasi
local domain $(D,M)$ such that $M^{-1}$~$=V$ a valuation ring. This makes
the maximal ideal of a non-valuation PVD a divisorial ideal.

\begin{corollary}
\label{Corollary J1} In a non-valuation PVD $D,$ every $v$-invertible ideal
is principal. Consequently a non-valuation PVD can never be a $v$-domain$.$
\end{corollary}

\begin{proof}
Suppose that a non-valuation PVD $D$ is a $v$-domain. Then $D$ is a $v$%
-local $v$-domain and hence a valuation domain by Proposition \ref%
{Proposition H}, a contradiction.
\end{proof}

\begin{remark}
\label{Remark J2} Using the fact that the set of prime ideals in a PVD is
linearly ordered it is shown in \cite{HH} that a GCD PVD is a valuation
domain. However a non-valuation PVD $D$ can never be a GCD domain, because a
GCD domain is a $v$-domain. We can also say that a non-valuation PVD can
never be a PVMD, because a PVMD is a $v$-domain as well.
\end{remark}

Let $S$ be a multiplicative set of $D.$ Following \cite{AAZ t-spl} we say
that $d\in D\backslash \{0\}$ is $t$-split by $S$ if there are two integral
ideals $A,B$ of $D$ such that $(d)=(AB)_{t}$ where $B_{t}\cap S\neq \phi $
and $(A,s)_{t}=D$ for all $s\in S.$ As in \cite{AAZ t-spl} we call $S$ a $t$%
-splitting set if $S$ $t$-splits every $d\in D\backslash \{0\}.$ By Lemma
2.1 of \cite{AAZ t-spl} if $S$ is a $t$-splitting set of $D,$ then $%
dD_{S}\cap D=A_{t}$ is a $t$-invertible $t$-ideal and hence a $v$-ideal and
of course $B_{t}=dA^{-1}.$

\begin{theorem}
\label{Theorem L}Let $S$ be a $t$-splitting set of an integral domain $D.$
If $D$ is a $v$-domain, then so is $D_{S}.$
\end{theorem}

\begin{proof}
Suppose that $D_{S}$ is not a $v$-domain. That is, there is a pair $a,b$ of $%
D_{S}$ such that $(((a,b)(a,b)^{-1})$ $D_{S})_{v}\neq D_{S}.$ Since $%
(r,s)^{-1}D_{S}=((r,s)D_{S})^{-1}$ for all $r,s\in D\backslash \{0\},$ we
can take $a,b\in D$ and regard $(a,b)(a,b)^{-1}$ as an ideal of $D.$ Since $%
(((a,b)(a,b)^{-1})$ $D_{S})_{v}\neq D_{S}$, $(a,b)(a,b)^{-1}\cap S=\phi .$
Again since $(((a,b)(a,b)^{-1})$ $D_{S})_{v}\neq D_{S}$ there exist $x,y\in
D_{S}$ such that $((a,b)(a,b)^{-1})$ $D_{S}\subseteq \frac{x}{y}D_{S}$ where 
$x\nmid y$ in $D_{S}$ and we can take $x,y$ in $D.$ This gives $%
y((a,b)(a,b)^{-1})D_{S}\subseteq xD_{S}$ and $y((a,b)(a,b)^{-1})$ $\subseteq
y((a,b)(a,b)^{-1})$ $D_{S}\cap D\subseteq xD_{S}\cap D.$ Now as $%
y((a,b)(a,b)^{-1})$ $\subseteq xD_{S}\cap D$ and $xD_{S}\cap D$ is
divisorial, we have $y((a,b)(a,b)^{-1})_{v}$ $\subseteq xD_{S}\cap D,$ which
forces $yD\subseteq xD_{S}\cap D.$ But then $yD_{S}\subseteq (xD_{S}\cap
D)D_{S}$ $=xD_{S}$ which contradicts the assumption that $x\nmid y$ in $%
D_{S}.$
\end{proof}

Let $X$ be an indeterminate over $D$, let $R=D[X]$ and let $G=\{f\in
D[X]|(A_{f})_{v}=D\}.$ It was shown in \cite[Proposition 3.7]{CDZ} that $G$
is a $t$-complemented $t$-lcm $t$-splitting set of $D[X]$. Here a $t$%
-splitting set $S$ is a $t$-lcm $t$-splitting set if for all $s\in S$ and
for all $x\in D\backslash \{0\},$ $sD\cap xD$ is $t$-invertible. The
following result was proved, as Theorem 3.4 in \cite{CDZ}.

\begin{proposition}
\label{Proposition L1}Let $D$ be an integral domain with quotient field $K$, 
$S$ a $t$-splitting set of $D$, and $\mathcal{S}=\{A_{1}$\textperiodcentered
\textperiodcentered \textperiodcentered $A_{n}|A_{i}=d_{i}DS\cap D$ for some 
$0\neq d_{i}\in D\}$. Then the following statements are equivalent. (1) $S$
is a $t$-lcm $t$-splitting set, (2) every finite type integral $v$-ideal of $%
D$ intersecting $S$ is $t$-invertible and (3) $D_{\mathcal{S}}$ $=\{x\in
K|xC\subseteq D$ for some $C\in T\}$ is a PVMD.
\end{proposition}

A $t$-splitting set $S$ is called $t$-complemented if $D_{\mathcal{S}}=D_{T}$
for some multiplicative set $T$ of $D.$

\begin{corollary}
\label{Corollary M}Let $X$ be an indeterminate over $D$, let $R=D[X]$ and
let $G=\{f\in D[X]|(A_{f})_{v}=D\}.$ Then $D$ is a $v$-domain if and only if 
$D[X]_{G}$ is.
\end{corollary}

\begin{proof}
Indeed as $D$ is a $v$-domain, then so is $D[X]$ \cite[Theorem 4.1]{FZ v}.
Since $G$ is a $t$-splitting set, Theorem \ref{Theorem L} applies. For the
converse, note that according to Proposition 3.7 of \cite{CDZ}, $G$ is a $t$%
-complemented $t$-lcm $t$-splitting set of $D[X].$ So, $D[X]_{\mathcal{S}}$
is a PVMD and there is a multiplicative set $N$ of $D[X]$ such that $D[X]_{%
\mathcal{S}}=D[X]_{N}.$ So $D[X]=D[X]_{G}\cap D[X]_{N}$ where $D[X]_{N}$ is
a PVMD. Thus if $D[X]_{G}$ is a $v$-domain, then so is $D[X].$ But then $D$
is a $v$-domain, \cite[Theorem 4.1]{FZ v}.
\end{proof}

Corollary \ref{Corollary M} can be put to an interesting use, but for that
we need some preparation. Let's first note that if $(D,M)$ is a $t$-local
domain and $X$ an indeterminate over $D,$ then $G=\{f\in
D[X]|(A_{f})_{v}=D\} $ is precisely $H=\{f\in D[X]|A_{f}=D\},$ because the
maximal ideal of $D$ is a $t$-ideal. In other words if $D$ is a $t$-local
domain, then $D[X]_{G}=D[X]_{H}=D(X)$, the Nagata extension of $D.$ For
description and properties of $D(X)$ the reader may consult \cite{AAM}.

\begin{corollary}
\label{Corollary N} (to Corollary \ref{Corollary M})Let D be a $t$-local
domain. Then $D$ is a $v$-domain if and only if $D(X)$ is a $v$-domain.
\end{corollary}

Next, according to Corollary 8 of \cite{BH}, if $\mathcal{P}$ is an
associated prime of a nonzero polynomial of $D[X],$ then $\mathcal{P\cap }%
D=(0)$ or $\mathcal{P}$\thinspace $=(\mathcal{P\cap }$ $D)[X]$ where $(%
\mathcal{P\cap }$ $D)$ is an associated prime of a principal ideal of $D.$

\begin{corollary}
\label{Corollary P}Let $D$ be an integral domain. Then $D$ is a super $v$%
-domain if and only if $D[X]$ is.
\end{corollary}

\begin{proof}
Let $D$ be a super $v$-domain. To see that $D[X]$ is a super $v$-domain let $%
\wp $ be an associated prime of $D[X]$. Then $\wp $ is an upper to $0,$
i.e., $\wp \cap D=(0)$ or $\wp =P[X]$ where $P$ is an associated prime of a
principal ideal of $D.$ If $\wp $ is an upper to $0$ then $D[X]_{\wp }$ is a
rank one DVR and so a $v$-domain. If, on the other hand, $\wp =P[X],$ where $%
P$ is an associated prime of a principal ideal of $D,$ then $D[X]_{\wp
}=D[X]_{P[X]}=D_{P}(X).$ Since $D$ is a a super $v$-domain, $D_{P}$ is a $v$%
-domain. But, then so is $D_{P}(X),$ by Corollary \ref{Corollary N}; because 
$D_{P}$ is $t$-local \cite[Corollary 2.3]{FZ t}. That $D[X]$ is a super $v$%
-domain, now follows from Theorem \ref{Theorem D}. For the converse note
that if $P$ is a minimal prime of $(a):(b)$ then $P[X]$ is minimal over $%
aD[X]:bD[X],$ making $P[X]$ an associated prime of a principal ideal of $%
D[X].$ Since $D[X]$ is a super $v$-domain, $D[X]_{P[X]}=D_{P}(X)$ is a $v$%
-domain. Now as $D_{P}$ is $t$-local, Corollary \ref{Corollary N} applies to
give the conclusion that $D_{P}$ is a a $v$-domain. Now $P$ being any
associated prime of $D$ we conclude, by Theorem \ref{Theorem D}, that $D$ is
indeed a super $v$-domain.
\end{proof}


\begin{thebibliography}{99}
\bibitem{AAM} \bigskip Anderson, D.D., Anderson, D.F. and Markanda, R.: The
rings $R(X)$ and $R<X>,$ J. Algebra 95 (1985) , 96-115.

\bibitem{AAFZ v} \bigskip Anderson, D.D., Anderson, D.F., Fontana, M. and
Zafrullah, M.: On $v$-domains and star operations, Comm. Algebra, 37 (2009)
3018--3043.

\bibitem{AAZ spl} \bigskip Anderson, D.D., Anderson, D.F. and Zafrullah, M.:
Splitting the $t$-class group,\ J. Pure Appl. Algebra 74(1991) 17-37.

\bibitem{AAZ t-spl} \bigskip Anderson, D.D., Anderson, D.F. and Zafrullah,
M.: The ring $D+XDS[X]$ and $t$-splitting sets, Commutative Algebra Arabian
J. Sci. Eng. Sect. C Theme Issues 26 (1) (2001) 3--16.

\bibitem{BH} Brewer, J. and Heinzer, W.: Associated primes of principal
ideals,\ Duke Math. J. 41(1974) 1-7.

\bibitem{CDZ} Chang, G.W., Dumitrescu, T. and Zafrullah, M.: $t$-Splitting
sets in integral domains, J. Pure Appl. Algebra 187 (2004) 71--86.

\bibitem{CMZ 1} Costa, D.L., Mott, J.L. and Zafrullah, M.: The construction $%
D+XD_{S}[X]$,\ J. Algebra 53(1978) 423-439.

\bibitem{CMZ 2} Costa, D.L., Mott, J.L. and Zafrullah, M.: Overrings and
dimensions of general $D+M$ constructions,\ J. Natur. Sci. and Math. 26 (2)
(1986), 7-14.

\bibitem{FZ v} Fontana, M. and Zafrullah, M.: On $v$-domains: a survey. In:
Fontana, M., Kabbaj, S., Olberding, B., Swanson, I. (eds.) Commutative
Algebra: Noetherian and Non-Noetherian Perspectives, pp. 145--180. Springer,
New York (2011)

\bibitem{FZ t} Fontana, M. and Zafrullah, M.: On $t$-local domains and
valuation domains, in "Advances in Commutative Algebra" Editors: Badawi,
Ayman, Coykendall, Jim, Trends in Mathematics, Birkh\"{a}user 2019, pp.
33-62.

\bibitem{Gab} Gabelli, S.: On divisorial ideals in polynomial rings over
Mori domains,\ Comm. Algebra 15(11)(1987) 2349-2370.

\bibitem{Gil} Gilmer, R.: Multiplicative Ideal Theory, Marcel-Dekker, New
York, 1972.

\bibitem{Gri} Griffin, M.: Some results on $v$-multiplication rings,\ Canad.
J. Math.19(1967) 710-722.

\bibitem{HH} Hedstrom, J. and Houston, E.: Pseudo-valuation domains, Pacific
J. Math. 75 (1978), 137--147.

\bibitem{He} Heinzer, W.: An essential integral domain with a nonessential
localization, Can. J. Math. 33, 400--403 (1981).

\bibitem{MNZ} Mott, J., Nashier, B., and Zafrullah, M.: Contents of
polynomials and invertibility. Comm. Algebra18 (1990), 1569--1583.

\bibitem{MZ} Mott, J. and Zafrullah, M.: On Pr\"{u}fer $v$-multiplication
domains,\ Manuscripta Math. 35(1981)1-26.

\bibitem{Nag 1} Nagata, M.: On Krull's conjecture concerning valuation
rings,\ Nagoya Math. J. (4)(1952) 29-33.

\bibitem{Nag 2} Nagata, M.: Correction to my paper \textquotedblright On
Krull's conjecture concerning valuation overrings,\ Nagoya Math.J. (9)(1955)
209-212.
\end{thebibliography}
\end{document}